\documentclass[11pt]{article}
\usepackage[utf8]{inputenc}

\usepackage{amsthm}
\usepackage{amsmath}
\usepackage{amssymb}
\usepackage{graphicx}
\usepackage{enumerate}
\usepackage{tikz}
\usetikzlibrary{math,patterns}
\input epsf.tex

\numberwithin{equation}{section}
\numberwithin{figure}{section}

\theoremstyle{plain}
\newtheorem{theorem}{\sffamily Theorem}
\newtheorem{proposition}{\sffamily Proposition}
\newtheorem{lemma}{\sffamily Lemma}
\newtheorem{corollary}{\sffamily Corollary}
\newtheorem{example}{\sffamily Example}
\newtheorem{remark}{\sffamily Remark}
\newtheorem{definition}{\sffamily Definition}
\newtheorem{conjecture}{\sffamily Conjecture}

\def\BET{\begin{theorem}}
\def\ENT{\end{theorem}}
\def\BEP{\begin{proposition}}
\def\ENP{\end{proposition}}
\def\BEL{\begin{lemma}}
\def\ENL{\end{lemma}}
\def\BEC{\begin{corollary}}
\def\ENC{\end{corollary}}
\def\BEE{\begin{example} \rm}
\def\ENE{\end{example}}
\def\BER{\begin{remark} \rm}
\def\ENR{\end{remark}}
\def\BED{\begin{definition} \rm}
\def\END{\end{definition}}
\def\BECJ{\begin{conjecture}}
\def\ENCJ{\end{conjecture}}

\def\bea{\begin{eqnarray}}
\def\eea{\end{eqnarray}}

\def\beq{\begin{equation}}
\def\eeq{\end{equation}}

\def\beal{\begin{align*}}

\def\eeal{ \end{align*} }

\def\bbR{{\mathbb R}}

\title{Existence of the discrete spectrum in the\\ Fichera layers and crosses of arbitrary dimension}
\author{
F.L. Bakharev\footnote{St.Petersburg State University, Universitetskaya emb. 7-9, St.Petersburg, 199034, Russia; E-mail: fbakharev@yandex.ru 
}
\ and
\setcounter{footnote}{6}
A.I. Nazarov\footnote{ St.Petersburg Department of Steklov Mathematical Institute of Russian Academy of Science, Fontanka 27, St.Petersburg, 191023, Russia, and St.Petersburg State University, Universitetskaya emb. 7-9, St.Petersburg, 199034, Russia; E-mail: nazarov@pdmi.ras.ru}
}

\begin{document}

\maketitle

\begin{abstract}
 We describe the Dirichlet spectrum structure for the Fichera layers and crosses
in any dimension $n\ge3$. Also the application of the obtained results to the classical Brownian exit times problem in these domains.
 \end{abstract}


{\bf AMS classification codes:} Primary: 35J05, 81Q10;
Secondary: 35K05, 60J65.

\medskip

\section{Introduction}
The existence of bound states in infinite regions with a hard-wall (or Dirichlet) boundary conditions belongs to trademark topics in spectral theory. A number of works was done in recent years about Dirichlet Laplacians in domains with tubular outlets to infinity (quantum waveguides). In such situations bound states (trapped modes) usually appear due to the geometrical structure of the domain in a finite region, however the reason of appearance may be different. The first possible reason is the presence of a sufficiently massive resonator, or, more carefully, the possibility to inscribe a sufficiently large body into the junction (see, e.g., \cite{NazSA, NaRuUu2013, Pa2017}). The second reason is geometrical bending without changing of the width (see, e.g., \cite{ESS, DuEx1995, GoJa1992}). In usual situation, a finite number of eigenvalues may appear under the threshold of the continuous spectrum. In some special cases it is possible to prove the uniqueness of such an eigenvalue (see, e.g., \cite{NazSA, NaRuUu2013, BaMaNa}). At the same time, including an additional small geometric parameter into the problem contributes to the construction of an arbitrarily large number of eigenvalues, either by enlarging the junction zone of the waveguides (see, e.g. \cite{DaRa2012, NaSh2014}), or by the localization effect near angular points (see, e.g., \cite{BaMaNaZaa}).

Because of greater variety in the geometric structure the unbounded Dirichlet layers and their junctions have been studied to a much less extent. As in the case of waveguides, layers do not contain arbitrarily large cubes, but do contain infinitely many disjoint cubes of a fixed size. Therefore, usually the continuous spectrum of the Dirichlet Laplacian is not empty and has a positive threshold. Only few results on the existence and properties of bound states under the threshold are known in this case, see \cite{CaExKr2004, DuExKr2001, LiLu2007, LuRo2012} and \cite[Ch. 4]{EK}. In particular, the existence of an infinite number of eigenvalues was established in the conical \cite{ExTa2010, DaOuRa2015, OuPa2018} and parabolic \cite{ExLo2020} layers. 

\medskip

In this paper we consider the spectral problem for the Dirichlet Laplacian
\begin{equation}
\label{A1}
 -\Delta u(x)= \lambda u(x), \quad x\in {\cal O};
 \qquad
u(x) = 0, \quad x\in \partial{\cal O} 
\end{equation}
in two domains ${\cal O}\subset \bbR^n$: 

1. ``Corner'' (see Fig.~\ref{fig-01} for $n=3$)
\begin{equation}
\label{corner}
\Theta^n_1=\left\{x\in \bbR^n\colon \Big|\min_{1\leq j\leq n} x_j\Big|<1\right\}.
\end{equation}

\begin{figure*}[ht!]
\centering
 \includegraphics[width=0.5\textwidth]{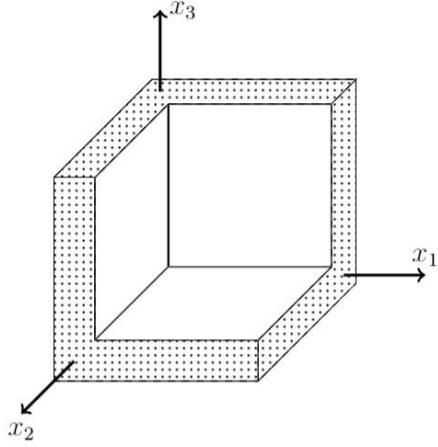}
\caption{Three-dimensional corner}
\label{fig-01}
\end{figure*}

2. ``Cross'' (see Fig.~\ref{fig-02} for $n=3$)
\begin{equation}
\label{cross}
\Theta^n_2=\left\{x\in \bbR^n\colon \min_{1\leq j\leq n} |x_j|<1\right\}.
\end{equation}

\begin{figure*}[ht!]
\centering
   \includegraphics[width=0.7\textwidth]{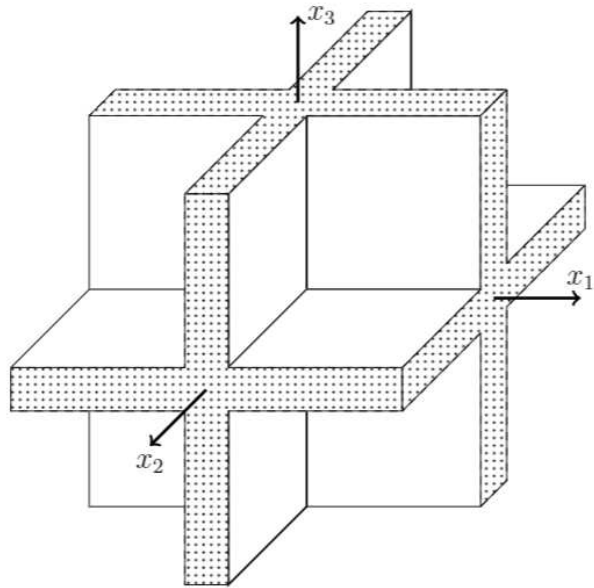}
\caption{Three-dimensional cross}
\label{fig-02}
\end{figure*}

For $n=2$ these domains are waveguides, and the problem \eqref{A1} is well studied, see, e.g., \cite{EK} and references therein. In both domains $\Theta^2_1$ and $\Theta^2_2$ the continuous spectrum of the problem \eqref{A1} coincides with the ray $[\frac {\pi^2}4,\infty)$, and there exists a unique eigenvalue $\lambda_\bullet<\frac {\pi^2}4$, see \cite{He}, \cite[Proposition 1.2.2]{EK} for the two-dimensional corner, \cite{SRW}, \cite[Proposition 1.5.2]{EK} for the two-dimensional cross, and also \cite{NazSA} for a more general setting. Numerical calculations give $\lambda_\bullet(\Theta^2_1)\approx 0.929\cdot \frac {\pi^2}4$, see, e.g., \cite{ESS}, \cite[Proposition 1.2.3]{EK}, and $\lambda_\bullet(\Theta^2_2)\approx 0.66\cdot \frac {\pi^2}4$,
see, e.g., \cite[Proposition 1.5.2]{EK}.

The case $n=3$ was considered in \cite{DaLaOu} where the domain $\Theta_1^3$ was named the Fichera layer. For both domains ${\cal O}=\Theta_j^3$ ($j=1,2$) it was proved that the problem \eqref{A1} has at most finite number of eigenvalues below the continuous spectrum. However, the existence of an eigenvalue was supported only by computation, without a theoretical proof.

We prove the existence of the discrete spectrum for the problem \eqref{A1} for ${\cal O}=\Theta^n_1$ and ${\cal O}=\Theta^n_2$ in any dimension $n\ge3$. Then we apply the obtained results to the so-called Brownian exit times problem in these domains. For some classes of convex unbounded domains this problem is well studied, see, e.g., \cite{BaDebS,BaS,WLi,LifZS}. For $\Theta^2_1$ and $\Theta^2_2$ it was considered in \cite{LN}.

\begin{remark}
The quantity of eigenvalues below the continuous spectrum in ${\cal O}=\Theta^n_1$ and ${\cal O}=\Theta^n_2$ for $n\ge3$ is unknown. We conjecture that similarly to the case $n=2$ there exists a unique eigenvalue in arbitrary dimension for both domains. 
\end{remark}

The structure of the paper is as follows. In Section \ref{prelim} we give the formal statement of the problem \eqref{A1} and collect the simplest properties of its bound state. Section \ref{contin} contains the description of the continuous spectrum of the problem, in Section \ref{discrete} we prove the existence of discrete spectrum. In Section \ref{teplo} we derive the asymptotic as $t\to\infty$ of the solution to an auxiliary initial-boundary value problem for the heat equation. Finally, Section \ref{exit} is devoted to the exit time problem.
\medskip

We use a standard notation
$$
(u,v)_{\cal O}=\int\limits_{\cal O}uv\,dx
$$
(for $u,v\in L_2({\cal O})$ it is just the scalar product).

$H^1_0({\cal O})$ stands for the subspace of all functions from the Sobolev space $H^1({\cal O})$ with zero trace on the boundary $\partial {\cal O}$.

Throughout the paper we write $\Theta^n$ if some argument is valid for both $\Theta^n_1$ and $\Theta^n_2$. 

Various constants depending only on $\Theta^n$ are denoted by $C$.

\section{Statement of the problem and preliminary information about the spectrum}
\label{prelim}

The problem \eqref{A1} admits variational formulation
\begin{equation}
\label{A3}
(\nabla u, \nabla v)_{\cal O}=\lambda (u,v)_{\cal O}, \quad \forall v\in H^1_0({\cal O}).
\end{equation}
Bilinear form on the left-hand side of \eqref{A3} is closed and positive and thus defines a self-adjoint  operator ${\mathfrak A}_{\cal O}$. Its spectrum $\sigma({\cal O})$ 
for the domains \eqref{corner}, \eqref{cross} is the main subject of the paper.

In what follows we denote by $\lambda_\bullet({\cal O})$ the lower bound of the spectrum $\sigma({\cal O})$, while $\lambda_\dagger({\cal O})$ stands for the lower bound of continuous spectrum $\sigma_c({\cal O})$. 

We mention that for a bounded domain ${\cal O}$ the spectrum is purely discrete and $\lambda_\dagger({\cal O})=+\infty$, while in the case of cylindrical domain ${\cal O}=Q\times \bbR$ ($Q\subset \bbR^{n-1}$) the spectrum is purely continuous and
\begin{equation}
\label{A6}
\lambda_\dagger({\cal O})=\lambda_\bullet({\cal O})=\lambda_\bullet(Q).
\end{equation}

Notice that the existence of a function $u\in H^1_0({\cal O})$ such that
\beq
\label{A5}
\|\nabla u; L_2({\cal O})\|^2-\lambda_\dagger({\cal O}) \,\|u; L_2({\cal O})\|^2<0,
\eeq
due to the variational principle (see, e.g., \cite[Sec. 10.2]{BS}) is equivalent to the inequality $\lambda_\bullet({\cal O})<\lambda_\dagger({\cal O})$ and guarantees the existence of at least one eigenvalue below the continuous spectrum. In this case $\lambda_\bullet({\cal O})$ is the smallest eigenvalue of the Dirichlet Laplacian in ${\cal O}$. Moreover, it is well known that the eigenvalue $\lambda_\bullet({\cal O})$ is simple and the corresponding eigenfunction $U_\bullet({\cal O})$ (the bound state) can be chosen positive. 

In Section \ref{discrete} below we prove the existence of discrete spectrum in $\Theta_1^n$ and in $\Theta_2^n$. The following Proposition provides the properties of corresponding bound states $U_j^n:=U_\bullet(\Theta_j^n)$, $j=1,2$.

\begin{proposition}\label{eigenfunction}
1. The functions $U_j^n$ are symmetric, i.e. for any permutation $\sigma\in {\cal S}_n$
$$
U_j^n(x_1,\ldots,x_n)=U_j^n(x_{\sigma(1)},\ldots,x_{\sigma(n)}).
$$
The function $U_2^n$ is also even w.r.t. any variable.

2. $U_j^n$ decay exponentially at infinity, i.e. there exists $\alpha_n>0$ such that
$$
|U_j^n(x)|+|\nabla_x U_j^n(x)|\leq C e^{-\alpha_n |x|},\quad x\in \bbR^n.
$$
\end{proposition}

\begin{proof}
The first assertion follows from simplicity of the corresponding eigenvalue. The second one was proved even for a more general case in \cite{Shn}, see also \cite[\S53]{Gl}.
\end{proof}

\section{Continuous spectrum in $\Theta^n_1$ and $\Theta^n_2$}
\label{contin}

First we prove an auxiliary statement.

\begin{lemma}
Let $n\ge3$.
For any $\varepsilon>0$ the following inequality holds:
\begin{equation}
\label{almost}
\|\nabla u; L_2(\Theta^n)\|^2\ge \big(\lambda_\bullet(\Theta^{n-1})-\varepsilon^2\big) \,\|u; L_2(\Theta^n)\|^2 - C \|u; L_2(\Theta^n\cap B^n_{\varepsilon^{-1}})\|^2.
\end{equation}
Here $u\in H^1_0(\Theta^n)$, and $C$ does not depend on $u$ and $\varepsilon$.
\end{lemma}

\begin{proof} We mainly follow the line of proof of \cite[Lemma 2.3]{K}.
Let $\rho_1$ and $\rho_2$ be smooth cutoff functions of $r=|x|$ such that 
$$
\rho_1(r)=0 \quad \text{for} \quad r<n+1, \qquad \rho_2(r)=0 \quad \text{for} \quad r>n+2, \qquad \rho_1^2+\rho_2^2=1.
$$ 
Then we derive
\begin{equation}
\label{ident}
\|\nabla u; L_2(\Theta^n)\|^2=\sum\limits_{k=1,2}\Big(\|\nabla (u\rho_k); L_2(\Theta^n)\|^2-\int\limits_{\Theta^n}u^2\Delta\rho_k\rho_k\,dx\Big),
\end{equation}
that implies
\begin{equation}
\label{ineq}
\|\nabla u; L_2(\Theta^n)\|^2\ge\|\nabla (u\rho_2); L_2(\Theta^n)\|^2-C_1\|u; L_2(\Theta^n\cap B^n_{n+2})\|^2.
\end{equation}
Denote for brevity $w=u\rho_2$.

Now we consider $n$ sets on the unit sphere
$$
A_k=\{x\in \mathbb{S}^{n-1}\colon \max_{j\ne k} |x_j|<|x_k| \},\qquad k=1,\dots,n
$$
and introduce the ``angular'' partition of unity, setting
$$
\wp_k^2(x)\equiv\wp_k^2(\Phi)=\int\limits_{A_k}\omega(\Phi-\Psi)\,dS_\Psi,\qquad \Phi=x/r\in\mathbb{S}^{n-1}
$$
(here $\omega$ is a mollifier with small given support).

So, $\sum\limits_{k=1}^n\wp_k^2(x)\equiv1$, and similarly to \eqref{ident} we derive
\begin{equation}
\label{ident1}
\|\nabla w; L_2(\Theta^n)\|^2=\sum\limits_{k=1}^n\Big(\|\nabla (w\wp_k); L_2(\Theta^n)\|^2-\int\limits_{\Theta^n}w^2\Delta\wp_k\wp_k\,dx\Big).
\end{equation}
Consider, for instance, the term corresponding to $k=n$. By the choice of cutoff functions $\rho_2$ and $\wp_n$, the support of $w\wp_n$ is contained in the set $\Theta^{n-1}\times \{|x_n|>1\}$. Therefore, we can extend $w\wp_n$ by zero to the cylindrical domain ${\cal O}=\Theta^{n-1}\times \bbR$, and the relation \eqref{A6} gives
$$
\|\nabla (w\wp_n); L_2(\Theta^n)\|^2\ge\lambda_\bullet(\Theta^{n-1})\,\|w\wp_n; L_2(\Theta^n)\|^2.
$$
Furthermore, since $\wp_n$ depends only on $x/r$, we have $|\Delta\wp_n|\le Cr^{-2}$.

Other terms in \eqref{ident1} are estimated in the same way. We sum up obtained inequalities and arrive at
\begin{align*}
\|\nabla w; L_2(\Theta^n)\|^2\ge & \,\lambda_\bullet(\Theta^{n-1})\,\|w; L_2(\Theta^n)\|^2 -C_2\|r^{-1}w; L_2(\Theta^n)\|^2\\
\ge & \,(\lambda_\bullet(\Theta^{n-1})-\varepsilon^2)\,\|w; L_2(\Theta^n)\|^2 -C_2\|w; L_2(\Theta^n\cap B^n_{\varepsilon^{-1}})\|^2.
\end{align*}
This gives \eqref{almost} in view of \eqref{ineq}.
\end{proof}

\begin{theorem} \label{spectr-contin}
The following relation holds for any $n\ge3$:
$$
\sigma_c(\Theta^n)=\big[\lambda_\bullet(\Theta^{n-1}),\infty\big).
$$
In particular, $\lambda_\dagger(\Theta^n)=\lambda_\bullet(\Theta^{n-1})$.
\end{theorem}

\begin{proof}
To prove that $\lambda_\dagger(\Theta^n)\geq\lambda_\bullet(\Theta^{n-1})$, following \cite[Sec. 19]{Gl} we proceed by contradiction. Let $\lambda<\lambda_\bullet(\Theta^{n-1})$ be a point of continuous spectrum. According to
\cite[Sec. 9.1]{BS}, there is a sequence $u_k\in Dom({\mathfrak A}_{\Theta^n})$ orthonormal in $L_2(\Theta^n)$ and satisfying the relation
$$
{\mathfrak A}_{\Theta^n}u_k-\lambda u_k\to0 \quad\text{in}\quad L_2(\Theta^n).
$$
Therefore, $\|\nabla u_k; L_2(\Theta^n)\|^2\to\lambda$. 

On the other hand, we can choose $\varepsilon$ such that 
$\lambda_\bullet(\Theta^{n-1})-\varepsilon^2>\lambda$. Since $u_k$ are orthonormal, it converges weakly to zero in $L_2(\Theta^n)$. Since $\|\nabla u_k; L_2(\Theta^n)\|$ are bounded, there is a subsequence $u_{k_m}\to0$ in $L_2(\Theta^n\cap B^n_{\varepsilon^{-1}})$. Therefore, \eqref{almost} implies
$$
\liminf \|\nabla u_{k_m}; L_2(\Theta^n)\|^2\ge \lambda_\bullet(\Theta^{n-1})-\varepsilon^2.
$$
This contradiction proves the inclusion $\sigma_c(\Theta^n)\subset\big[\lambda_\bullet(\Theta^{n-1}),\infty\big)$. 

The opposite inclusion can be verified in a standard way using appropriate Weyl sequences (cf. \cite[Theorem 9.1.2]{BS}). For the set $\Theta_2^3$ it was done in \cite[Theorem 1.2]{DaLaOu}.
\end{proof}

\section{Existence of discrete spectrum}
\label{discrete}

\subsection{Discrete spectrum in $\Theta^n_1$}

\begin{theorem}\label{discrete1}
The following relation holds:
\begin{equation}\label{spectr-corner}
\lambda_\bullet(\Theta^n_1)<\lambda_\dagger(\Theta^n_1).
\end{equation}
\end{theorem}

\begin{proof}
We use induction in $n$. As we mentioned in the Introduction, the case $n=2$ is well known. Let us prove the statement for $n=3$. We denote for simplicity $U
=U_\bullet(\Theta_1^2)$ and normalize it by condition 
$
\|U;L_2(\Theta^{2}_1)\|=1.
$

We introduce the notation
\begin{align*}
&\Pi^n_k=\left\{x\in \Theta^n_1\colon \max_{j\ne k} x_j<x_k \right\},\\
&\Gamma^n_{jk}=\Gamma^n_{kj}=\partial \Pi^n_k\cap \partial \Pi^n_j,
\end{align*}
see Fig.~\ref{fig-03} for $n=3$.

\begin{figure*}[ht!]
\centering
   \includegraphics[width=0.7\textwidth]{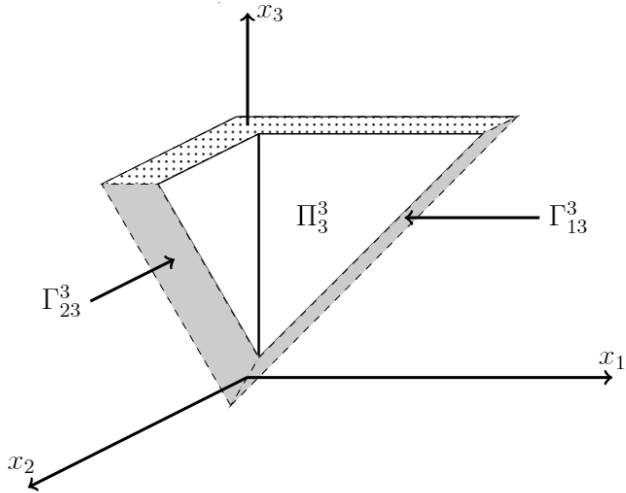}
\caption{The set $\Pi_3^3$}
\label{fig-03}
\end{figure*}
\medskip

We introduce a function $W_\beta$ in a following way: for $x=(x',x_3)\in \overline{\Pi^3_3}$ we put
$$
W_\beta(x)=U(x')e^{-\beta x_3},
$$
and define it in a similar way in $\overline{\Pi^3_k}$ for $k=1, 2$. Due to the symmetry of $U$, see Proposition \ref{eigenfunction}, the function $W_\beta$ falls into the space $H^1_0(\Theta^3_1)$ for any $\beta>0$. 

We proceed with a transformation of the left-hand side of the inequality \eqref{A5} with $u=W_\beta$, ${\cal O}=\Theta^3_1$, $\lambda_\dagger(\Theta^3_1)=\lambda_{\bullet}(\Theta^{2}_1)$ (the last equality follows from Theorem \ref{spectr-contin}). Due to the symmetry of $W_\beta(x)$, it is sufficient to calculate integrals only over $\Pi^3_3$:
$$
\aligned
I(\beta):= &\, \|\nabla W_\beta;L_2(\Pi_3^3)\|^2-\lambda_\bullet(\Theta^{2}_1)\|W_\beta;L_2(\Pi_3^3)\|^2=\\
= &\, \beta^2\|W_\beta;L_2(\Pi_3^3)\|^2 + \int\limits_{\Pi^3_3} |\nabla U (x')|^2 e^{-2\beta x_3 }\, dx-\lambda_\bullet(\Theta^{2}_1)\|W_\beta;L_2(\Pi_3^3)\|^2\\
\stackrel{*}= &\, \beta^2\|W_\beta;L_2(\Pi_3^3)\|^2 + \int\limits_{\partial \Pi_3^3} \nu\cdot \nabla U (x') U (x') e^{-2\beta x_3 }\,dS_{2}(x)
\endaligned
$$
where $\nu
$ is the unit outward normal to the boundary $\partial \Pi_3^3$ and $S_{2}$ is corresponding surface area. The equality (*) follows from integration by parts in the second term and the equation for the eigenfunction $U$. 

Denote by $J(\beta)$ the last integral over $\partial \Pi_3^3$. Due to the boundary condition for $U$, this integral is taken in fact over $\Gamma^3_{31}\cup\Gamma^3_{32}$. Using the symmetry of $U$ and relations $\nu=\big(\frac{1}{\sqrt{2}},0,-\frac{1}{\sqrt{2}}\big)$, $x_1=x_3$ on $\Gamma^3_{31}$ we rewrite $J(\beta)$ as follows:
$$
\aligned
J(\beta)= &\, 2\int\limits_{\Gamma^3_{31}} \frac{1}{\sqrt{2}}\, D_{x_1} U (x') U (x') e^{-2\beta x_3}\,dS_{2}(x)\\
= &\, 2\int\limits_{\Pi^{2}_1} D_{x_1} U (x') U (x') e^{-2\beta x_1}\,dx'.
\endaligned
$$
We again integrate by parts and use the boundary condition for $U$ and the relation $\nu=\big(-\frac{1}{\sqrt{2}},\frac{1}{\sqrt{2}}\big)$ for the outward normal on $\Gamma^2_{12}$. This gives
$$
J(\beta)=2\beta \int\limits_{\Pi^{2}_1} |U (x')|^2 e^{-2\beta x_1}d x'-\int\limits_{\Gamma^{2}_{12}} \frac{1}{\sqrt{2}}|U(x')|^2e^{-2\beta x_1}\, dS_{1}(x').
$$
Finally we arrive at
\begin{multline*}
I(\beta)=-\frac{1}{\sqrt{2}}\int\limits_{\Gamma^{2}_{12}} |U(x')|^2e^{-2\beta x_1}\, dS_{1}(x')  +
\\+ 2\beta  \int\limits_{\Pi^{2}_1} |U (x')|^2 e^{-2\beta x_1}d x' + \beta^2\|W_\beta;L_2(\Pi_3^3)\|^2\,.
\end{multline*}
Now we push $\beta\to+0$ and observe that the Lebesgue dominated convergence theorem ensures 
$$
\aligned
&\int\limits_{\Gamma^{2}_{12}} |U(x')|^2e^{-2\beta x_1}dS_{1}(x')\to \int\limits_{\Gamma^{2}_{12}} |U(x')|^2dS_{1}(x') >0,\\
2\beta &\, \int\limits_{\Pi^{2}_1} |U (x')|^2 e^{-2\beta x_1}d x'\to 0,
\endaligned
$$
(here we use the exponential decay of $U$, see Proposition \ref{eigenfunction}).
The last term in $I(\beta)$ goes to zero due to the estimate
$$
\beta^2\int\limits_{\Pi_3^3} |U (x')|^2e^{-2\beta x_3 } dx\leq \beta^2\int\limits_0^{+\infty}e^{-2\beta x_3 } \,dx_3 = \frac{\beta}{2}. 
$$
Therefore, $I(\beta)$ is negative for sufficiently small $\beta$. This proves \eqref{spectr-corner} for $n=3$ and ensures the existence of an eigenfunction $U_\bullet(\Theta_1^3)$. This fact, in turn, allows us to manage the same proof for $n=4$ etc. 
The only difference is the following: the set $\Gamma^{2}_{12}$ is just a segment, while for $n>3$ the set $\Gamma^{n-1}_{12}$ is unbounded. Nevertheless the integral
$$
\int\limits_{\Gamma^{n-1}_{12}} |U(x')|^2\,d S_{n-2}(x') 
$$
is still finite due to the exponential decay of $U$.
\end{proof}

\begin{remark}
Theorem~\ref{discrete1} guarantees that the sequence $\{\lambda_\bullet(\Theta_1^n)\}_{n=2}^\infty$ is strictly decreasing. Since it is positive the natural question is whether the limit is zero. The answer is negative. 
\end{remark}

\begin{theorem}\label{bounded-from-zero1}
For any $n\geq 2$, we have $\lambda_\bullet(\Theta_1^n)\geq \pi^2/16$.
\end{theorem}

\begin{proof}
We split the set $\Theta_1^n$ into $n$ parts
$$
\widetilde \Pi^n_k=\left\{x\in \Theta^n_1\colon \min_{j\ne k} x_j>x_k \right\},
$$ 
see Fig.~\ref{fig-04} for $n=3$.

\begin{figure*}[ht!]
\centering
\includegraphics[width=0.8\textwidth]{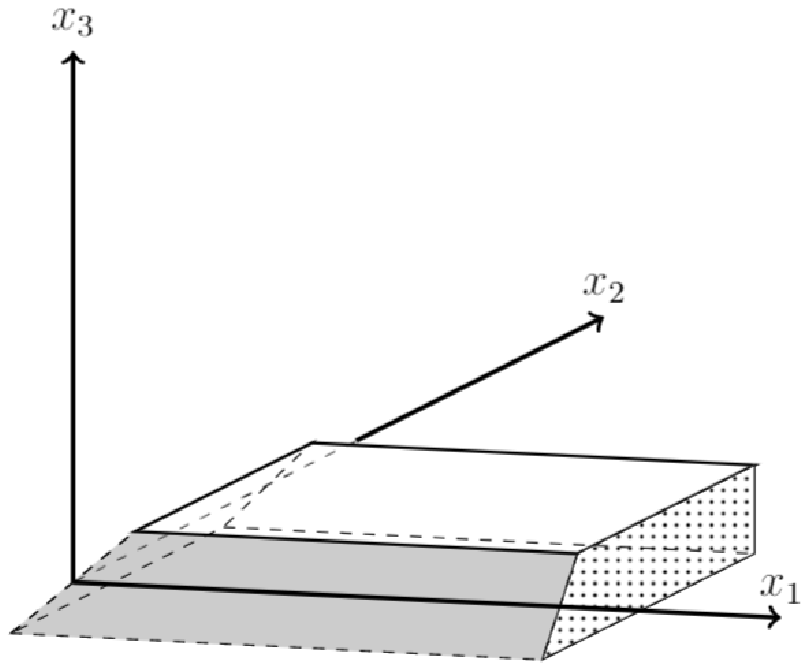}
\caption{The set $\widetilde\Pi_3^3$}
\label{fig-04}
\end{figure*}

For any $u\in H^1_0(\Theta_1^n)$ we have $u|_{x_n=-1}=0$. So, the classical Steklov inequality gives
$$
\int\limits_{-1}^{a}u^2(x',x_n)\,dx_n\le \Big(\frac {2}{\pi}\,(a+1)\Big)^2\int\limits_{-1}^{a}(D_{x_n}u)^2(x',x_n)\,dx_n
$$
for $a\in(-1,1)$. Integrating over $x'$ provides us the inequality
$$
\|\nabla u;L_2(\widetilde\Pi_n^n)\|^2\geq\|D_{x_n} u;L_2(\widetilde\Pi_n^n)\|^2\geq \frac{\pi^2}{16}  \|u;L_2(\widetilde\Pi_n^n)\|^2. 
$$
The same estimate is valid for other $\widetilde\Pi_k^n$, and the statement follows.
\end{proof}

\subsection{Discrete spectrum in $\Theta^n_2$}

The inequality $\lambda_\bullet(\Theta^n_2)<\lambda_\dagger(\Theta^n_2)$ can be proved by the same method as \eqref{spectr-corner}. We prefer to prove a slightly more general statement dealing with arbitrary perturbed cylinder, cf. \cite{PT}. 

\begin{theorem}\label{cylinder-with-ears}
Let ${\cal O}\supsetneq(Q\times\bbR)$ where $Q\subset\bbR^{n-1}$  but  ${\cal O}\cap \{x:\,|x_n|>R\}= Q\times\{|x_n|>R\}$. Finally, suppose that there is a bound state $U_\bullet(Q)$. Then $\lambda_\bullet({\cal O})<\lambda_\bullet(Q)$.
\end{theorem}

\begin{proof}
In the case where all cross-sections ${\cal O}\cap \{x:\,x_n=a\}$ are bounded (in particular, $Q$ is bounded) this statement is well known, see, e.g., \cite[Theorem 1.4]{EK}. However, the proof in \cite{EK} is not applicable in our case.

We normalize $U_\bullet(Q)$ by condition $\|U_\bullet(Q);L_2(Q)\|=1$.

First, we consider the domain 
$$
{\cal O}_R={\cal O}\cap \{x:\,|x_n|<R\}\supsetneq Q\times (-R,R)
$$
and define the Rayleigh quotient
$$
J_R(u)=\frac {\|\nabla u; L_2({\cal O}_R)\|^2}{\|u; L_2({\cal O}_R)\|^2}
$$
on the set
$$
M_R=\{u\in H^1({\cal O}_R)\,: \, u\big|_{Q\times \{-R,R\}}=U_\bullet(Q);\ \ u\big|_{\partial{\cal O}_R\setminus(Q\times \{-R,R\})}=0\}.
$$
We claim that $\inf\limits_{u\in M_R}J_R(u)<\lambda_\bullet(Q)$. Indeed, the function
$$
u_R(x)\equiv u_R(x',x_n)=\begin{cases}
U_\bullet(Q)(x'), & x'\in Q;\\
0,& x'\notin Q
\end{cases}
$$
evidently falls into $M_R$, and $J_R(u_R)=\lambda_\bullet(Q)$. However, if $u_R$ minimizes $J_R$ then it should satisfy the Euler-Lagrange equation $-\Delta u_R=\lambda_\bullet(Q)u_R$ in ${\cal O}_R$. This is impossible by the maximum principle, and the claim follows.

Thus, there is a function $\widehat u_R\in M_R$ such that $J_R(\widehat u_R)<\lambda_\bullet(Q)$. Denote $A=\|\widehat u_R; L_2({\cal O}_R)\|^2$.

Now we define a function in $H^1_0({\cal O})$ (here $\beta>0$)
$$
\widetilde W_\beta(x)\equiv \widetilde W_\beta(x',x_n)=\begin{cases}
U_\bullet(Q)(x')e^{-\beta (|x_n|-R)}, & |x_n|>R;\\
\widehat u_R(x),& |x_n|\le R.
\end{cases}
$$
Direct calculation gives
$$
\|\widetilde W_\beta; L_2({\cal O})\|^2=A+\frac 1\beta;\qquad \|\nabla \widetilde W_\beta; L_2({\cal O})\|^2=A\cdot J_R(\widehat u_R)+\frac {\lambda_\bullet(Q)}\beta+\beta,
$$
and therefore
$$
\|\nabla \widetilde W_\beta; L_2({\cal O})\|^2-\lambda_\bullet(Q) \,\|\widetilde W_\beta; L_2({\cal O})\|^2=A\cdot(J_R(\widehat u_R)-\lambda_\bullet(Q))+\beta,
$$
that is negative for sufficiently small $\beta$.
\end{proof}

\begin{corollary}\label{discrete2}
For any $n\ge2$, the following relation holds:
\begin{equation}\label{spectr-cross}
\lambda_\bullet(\Theta^n_2)<\lambda_\dagger(\Theta^n_2).
\end{equation}
\end{corollary}

\begin{proof}
Similarly to Theorem \ref{discrete1}, we use induction in $n$. As it was noticed in the Introduction, the inequality \eqref{spectr-cross} for $n=2$ is known long ago. Therefore, there is a bound state $U_\bullet(\Theta^2_2)$, and we can apply Theorem \ref{cylinder-with-ears} to $\Theta^3_2$. Theorems \ref{spectr-contin} and \ref{cylinder-with-ears} ensure \eqref{spectr-cross}
for $n=3$, etc.
\end{proof}

\begin{remark}\label{bounded-from-zero2}
As for $\Theta_1^n$, we have a strictly decreasing sequence $\{\lambda_\bullet(\Theta_2^n)\}_{n=2}^\infty$. We conjecture that similarly to Theorem \ref{bounded-from-zero1} its limit is strictly positive. However, this problem is completely open.

\end{remark}

\section{The heat equation in $\Theta^n_1$ and $\Theta^n_2$}
\label{teplo}

In this Section for the sake of brevity we denote by $\lambda^n=\lambda_\bullet(\Theta^n)$ the least eigenvalue of ${\mathfrak A}_{\Theta^n}$. Also we denote by $U^n$ corresponding bound state and by $\Lambda^n$ the infimum of other spectrum $\sigma(\Theta^n)\setminus\{\lambda^n\}$ (note that $\Lambda^n$ either coincides with $\lambda_\dagger(\Theta^n)$ or is another eigenvalue).
\medskip

Consider the initial-boundary value problem
\begin{equation} \label{heat}
   {\cal L} v:=\partial_tv-\frac 12\Delta v=0
   \quad\mbox{in}\quad {\cal O}\times \bbR_+; \qquad  v\big|_{x\in\partial{\cal O}}=0;
   \qquad v\big|_{t=0}\equiv 1.
\end{equation}

\begin{theorem}
\label{asymp-heat} Let ${\cal O}=\Theta^n$. Then the solution $V^n$ of the problem \eqref{heat} admits a representation
\begin{equation} \label{asymp1}
\aligned
  V^n(x;t)= &\, {\cal A}_n \, U^n(x)\exp\big(-\frac {\lambda^n}2 t\big)
   +v^n(x;t),\\ 
   |v^n(x;t)|\le &\, C(t+1)\exp\big(-\frac {\Lambda^n}2 t\big),
\endaligned
\end{equation}
where ${\cal A}_n = \int\limits_{\Theta^n} U^n(x)\,dx$.
\end{theorem}

\begin{proof}
We again use induction in $n$. The case $n=2$ was considered in \cite{LN}. Along with \eqref{asymp1}, it was in fact proved that the difference
\begin{equation}\label{difference}
\widehat v^2(x_1,x_2;t):=V^2(x_1,x_2;t)-V^1(x_2;t)\chi(x_1)-V^1(x_1;t)\chi(x_2)
\end{equation}
is rapidly decaying as either $|x_1|\to\infty$ or $|x_2|\to\infty$. Here 
\begin{equation}\label{heat1}
V^1(y;t) =\frac{4}{\pi}\, \sum_{k=0}^\infty
    \frac{(-1)^k}{2k+1}\, \cos\left(\pi(k+1/2)y\right) \exp({-(k+1/2)^2\pi^2t/2})
\end{equation}
is the well-known solution of the problem \eqref{heat} for $n=1$, $\Theta^1=(-1,1)$, while
$\chi$ is a smooth cutoff function such that 
$$
   \chi(x_1)=0
    \ \ \mbox{for} \ \ |x_1|<2;
    \quad \chi(x_1)=1\ \ \mbox{for}
    \ \ |x_1|>3.
$$

We seek for the solution $V^3$ in the form
$$
\aligned
V^3(x;t)= &\, V^2(x_1,x_2;t)\chi(x_3)+V^2(x_2,x_3;t)\chi(x_1)+V^2(x_3,x_1;t)\chi(x_2)\\
- &\, V^1(x_1;t)\chi(x_2)\chi(x_3)-V^1(x_2;t)\chi(x_3)\chi(x_1)\\
-&\,V^1(x_3;t)\chi(x_1)\chi(x_2)+\widehat v^3(x_1,x_2,x_3;t).
\endaligned
$$
Then the correction term $\widehat v^3$ is the solution of the problem
\begin{equation*} 
   {\cal L} \widehat v^3=f
   \quad\mbox{in}\quad \Theta^3\times \bbR_+; \qquad  v\big|_{x\in\partial\Theta^3}=0;
   \qquad\widehat v^3\big|_{t=0}=\varphi \quad \mbox{in}\quad \Theta^3,
\end{equation*}
where
$$
\aligned
f(x;t):= &\, \frac 12 \big(V^2(x_1,x_2;t)-V^1(x_2;t)\chi(x_1)-V^1(x_1;t)\chi(x_2)\big)D^2\chi(x_3)\\
+ &\, \frac 12 \big(V^2(x_2,x_3;t)-V^1(x_3;t)\chi(x_2)-V^1(x_2;t)\chi(x_3)\big)D^2\chi(x_1)\\
+ &\, \frac 12 \big(V^2(x_3,x_1;t)-V^1(x_1;t)\chi(x_3)-V^1(x_3;t)\chi(x_1)\big)D^2\chi(x_2);\\
\varphi(x):= &\, (1-\chi(x_1))(1-\chi(x_2))(1-\chi(x_3)).
\endaligned
$$
Notice that due to the properties of the function \eqref{difference} the right-hand side $f$ is rapidly decaying as $|x|\to\infty$ while the initial data $\varphi$ is compactly supported. Therefore, we can apply the spectral decomposition for the operator exponent:
\begin{equation*}
    \aligned
&    \widehat v^3(\cdot\,;t) = \exp\big(\frac 12\Delta t\big)\varphi(\cdot)
     + \int\limits_0^t\exp\big(\frac 12\Delta(t-s)\big)f(\cdot\,;s)\,ds
    \\
&    =\! \int\limits_{[\lambda^3,\infty)}\!\! \exp\big(-\frac {\lambda}2 t\big)
      \,dE(\lambda)\varphi(\cdot)
      +\int\limits_0^t\int\limits_{[\lambda^3,\infty)}\!\!\exp\big(-\frac {\lambda}2 (t-s)\big)
      \,dE(\lambda)f(\cdot\,;s)\,ds,
      \endaligned
\end{equation*}
where $E(\lambda)$ is the projector-valued spectral measure generated
by the (self-adjoint) Dirichlet Laplacian operator in $L_2(\Theta^3)$,
see \cite[Ch. 6]{BS}.

Since $\lambda^3$ is a simple eigenvalue, we can rewrite the latter equality
as $\widehat v^3={\mathfrak v}_0+{\mathfrak v}_1$, where
\begin{equation*}
    \aligned
    {\mathfrak v}_0(\cdot\,;t) = &\, \Big(\exp\big(-\frac {\lambda^3}2 t\big)
    \big(\varphi,U^3\big)_{\Theta^3}\\
    + & \int\limits_0^t\exp\big(-\frac {\lambda^3}2 (t-s)\big)
    \big(f(\cdot\,;s),U^3\big)_{\Theta^3}\,ds\Big) \, U^3(\cdot);
    \endaligned
\end{equation*}
\begin{equation*}
    \aligned
    {\mathfrak v}_1(\cdot\,;t)= &\, {\mathfrak v}_{11}(\cdot\,;t)+{\mathfrak v}_{12}(\cdot\,;t)
    :=\int\limits_{[\Lambda^3,\infty)}\exp\big(-\frac {\lambda}2 t\big)\,dE(\lambda)
    \varphi(\cdot)\\
    + & \int\limits_0^t\int\limits_{[\Lambda^3,\infty)}\!
      \exp\big(-\frac {\lambda}2 (t-s)\big)\,dE(\lambda)f(\cdot\,;s)\,ds.
    \endaligned
\end{equation*}

First, we estimate $L_2$-norm of ${\mathfrak v}_1$. Estimate
of the first term is obvious:
$$
	\|{\mathfrak v}_{11}(\cdot\,;t); L_2(\Theta^3)\|^2
    \le\exp(-\Lambda^3 t)\,\|\varphi; L_2(\Theta^3)\|^2.
$$
Since $|f(\cdot\,;s)|<C\exp(-\frac {\lambda^2}2 s)$ and $f(\cdot\,;s)$ 
is uniformly rapidly decaying for $s\in[0,t]$, the estimate for
${\mathfrak v}_{12}$ follows from $\Lambda^3\le\lambda^2$:
$$
    \aligned
	\|{\mathfrak v}_{12}(\cdot\,;t); L_2(\Theta^3)\|^2
	\le &\, C\exp(-\Lambda^3 t)\,
	\Big(\int\limits_0^t\exp\big(\frac {\Lambda^3-\lambda^2}2\, s\big)\,ds\Big)^2\\
    \le &\, Ct^2\exp(-\Lambda^3 t).
    \endaligned
$$
Thus, we have $\|{\mathfrak v}_1(\cdot\,;t); L_2(\Theta^3)\|\le C(t+1)\exp(-\frac {\Lambda^3}2 t)$.
Therefore, standard parabolic estimates give
$|{\mathfrak v}_1(\cdot\,;t)|\le C(t+1)\exp(-\frac {\Lambda^3}2 t)$.\medskip

Next, we consider the second term in ${\mathfrak v}_0$:
\begin{equation}\label{(f,U)}
    \aligned
& \big(f(\cdot\,;s),U^3\big)_{\Theta^3}=\frac 12\big(\widehat v^2(x_1,x_2;s)D^2\chi(x_3),U^3\big)_{\Theta^3} \\
+\frac 12& \,\big(\widehat v^2(x_2,x_3;s)D^2\chi(x_1),U^3\big)_{\Theta^3}+\frac 12\big(\widehat v^2(x_3,x_1;s)D^2\chi(x_2),U^3\big)_{\Theta^3}
    \endaligned
\end{equation}
(recall that the function $\widehat v^2$ is defined in \eqref{difference}).

We introduce the notation
$$
    \aligned
\widetilde v^2(x_1,x_2;t)= & \,\widehat v^2(x_1,x_2;t)+\frac 12 V^1(x_2;t)\chi(x_1)+\frac 12 V^1(x_1;t)\chi(x_2) \\
= & \, V^2(x_1,x_2;t)-\frac 12 V^1(x_2;t)\chi(x_1)-\frac 12 V^1(x_1;t)\chi(x_2).
    \endaligned
$$
Integration by parts with regard to the equations for $U^3$, $V^2$ and $V^1$ provides
\begin{equation*}
    \aligned
   & \big(\widetilde v^2(x_1,x_2;s)D^2\chi(x_3),U^3\big)_{\Theta^3}
   =\big(\widetilde v^2(x_1,x_2;s)\chi(x_3),D^2_{x_3}U^3\big)_{\Theta^3}
\\
   = -\lambda^3 &\,\big(\widetilde v^2(x_1,x_2;s)\chi(x_3),U^3\big)_{\Theta^3}
   -\big(\widetilde v^2(x_1,x_2;s)\chi(x_3),(D^2_{x_1}+D^2_{x_2})U^3\big)_{\Theta^3}
\\
   = -\lambda^3 &\,\big(\widetilde v^2(x_1,x_2;s)\chi(x_3),U^3\big)_{\Theta^3}
   -\big((D^2_{x_1}+D^2_{x_2})\widetilde v^2(x_1,x_2;s)\chi(x_3),U^3\big)_{\Theta^3}
\\
   = -\lambda^3 &\,\big(\widetilde v^2(x_1,x_2;s)\chi(x_3),U^3\big)_{\Theta^3}
   -2\big(\partial_s\widetilde v^2(x_1,x_2;s)\chi(x_3),U^3\big)_{\Theta^3}
\\
  + \frac 12 &\, \big((V^1(x_2;s)D^2\chi(x_1)+V^1(x_1;s)D^2\chi(x_2))\chi(x_3),U^3\big)_{\Theta^3}
    \endaligned
\end{equation*}
(notice that all integrals converge due to the exponential decay of $U^3$). 

Thus, the first scalar product in \eqref{(f,U)} is rewritten as follows:
\begin{equation*}
    \aligned
   & \big(\widehat v^2(x_1,x_2;s)D^2\chi(x_3),U^3\big)_{\Theta^3}
\\
   = -\lambda^3 &\,\big(\widetilde v^2(x_1,x_2;s)\chi(x_3),U^3\big)_{\Theta^3}
   -2\big(\partial_s\widetilde v^2(x_1,x_2;s)\chi(x_3),U^3\big)_{\Theta^3}
\\
  + \frac 12 &\, \big((V^1(x_2;s)D^2\chi(x_1)+V^1(x_1;s)D^2\chi(x_2))\chi(x_3),U^3\big)_{\Theta^3}
\\
  - \frac 12 &\, \big((V^1(x_2;s)\chi(x_1)+V^1(x_1;s)\chi(x_2))D^2\chi(x_3),U^3\big)_{\Theta^3},
    \endaligned
\end{equation*}
We rewrite other terms in \eqref{(f,U)} in a similar way and obtain
\begin{equation*}
    \aligned
  &\frac 12\int\limits_0^t\exp\big(-\frac {\lambda^3}2 (t-s)\big)
     \big(f(\cdot\,;s),U^3\big)_{\Theta^3}\,ds=
-\exp\big(-\frac {\lambda^3}2 t\big) \int\limits_0^t
    \exp\big(\frac {\lambda^3}2 s\big)\times
\\
    \times &\, \Big[\big(\frac {\lambda^3}2\widetilde v^2(x_1,x_2;s)\chi(x_3)
     +\partial_s\widetilde v^2(x_1,x_2;s)\chi(x_3),U^3\big)_{\Theta^3}
\\
    + &\, \big(\frac {\lambda^3}2\widetilde v^2(x_2,x_3;s)\chi(x_1)
     +\partial_s\widetilde v^2(x_2,x_3;s)\chi(x_1),U^3\big)_{\Theta^3}
\\
    + &\, \big(\frac {\lambda^3}2\widetilde v^2(x_3,x_1;s)\chi(x_2)
     +\partial_s\widetilde v^2(x_3,x_1;s)\chi(x_2),U^3\big)_{\Theta^3}\Big]\,ds
\\
  = &\, - \Big[
  \big(\widetilde v^2(x_1,x_2;s)\chi(x_3),U^3\big)_{\Theta^3}
  + \big(\widetilde v^2(x_2,x_3;s)\chi(x_1),U^3\big)_{\Theta^3}
\\
  &\, + \widetilde v^2(x_3,x_1;s)\chi(x_2),U^3\big)_{\Theta^3}\Big]
  \cdot\exp\big(-\frac {\lambda^3}2 (t-s)\big)\bigg|_{s=0}^{s=t}.
  \endaligned
\end{equation*}
The substitution at $s=t$ is $O\big(\exp(-\frac {\lambda^2}2 t)\big)$, whereas
$$
  \aligned
&\, \widetilde v^2(x_1,x_2;0)\chi(x_3)  + \widetilde v^2(x_2,x_3;0)\chi(x_1) + \widetilde v^2(x_3,x_1;0)\chi(x_2) \\
= &\, \chi(x_1)+\chi(x_2)+\chi(x_3)-\chi(x_1)\chi(x_2)-\chi(x_2)\chi(x_3)-\chi(x_3)\chi(x_1)
\\
= &\, 1-\varphi(x) \quad \mbox{in}\quad \Theta^3,
  \endaligned
$$

We add the first term in ${\mathfrak v}_0$ and arrive at
\begin{equation*}
   V^3(x;t)=\big(1,U^3\big)_{\Theta^3}\,U^3(x)
   \exp(-\frac {\lambda^3}2 t)
   +
   O\Big((t+1)\exp(-\frac {\Lambda^3}2 t)\Big).
\end{equation*}
This gives \eqref{asymp1} for $n=3$. 

Also it follows from the proof that $\widehat v^3(\cdot\,;t)$ is rapidly decaying as $|x|\to\infty$. This fact, in turn, allows us to manage the same proof for $n=4$ etc.
\end{proof}

\section{Application to the exit time problem}
\label{exit}

Let $W(t)=\big(W_j(t)\big)\big|_{j=1}^n$, $t\ge 0$, be a standard $n$-dimensional
Brownian motion with independent coordinates, $n\ge 1$, and let $W^x(t)=x+W(t)$
denote a Brownian motion starting from a point $x\in \bbR^n$.

We begin with two small deviation problems: find the asymptotic behavior of
\begin{eqnarray*}
\mathbb{P}\big\{ \| \min_{1\le j\le n} W_j\| \le \varepsilon \big\}, \qquad \varepsilon \to 0,
\\
  \mathbb{P}\big\{ \| \min_{1\le j\le n} |W_j|\| \le \varepsilon \big\} , \qquad \varepsilon \to 0.
\end{eqnarray*}
Here $\|\cdot\|$ stands for sup-norm, i.e. $\|X\|=\max\limits_{0 \le t \le 1} |X(t)|$.

These problems are immediately reduced to the so-called exit time asymptotics. Denote by
$\tau_{x,{\cal O}}$ the exit time for $W^x$ from a domain ${\cal O}\subset\bbR^n$ containing
a point $x$, and put $\tau_{\cal O}=\tau_{0,{\cal O}}$. Then, by self-similarity of $W$,
we have
\begin{eqnarray*}
   \mathbb{P}\big\{ \| \min_{1\le j\le n} W_j\| \le \varepsilon \big\}
   &=& \mathbb{P}\big\{ \tau_{\Theta^n_1} \ge \varepsilon^{-2}\big\},
\\
   \mathbb{P}\big\{ \| \min_{1\le j\le n} |W_j|\| \le \varepsilon\big\}
    &=& \mathbb{P}\big\{ \tau_{\Theta^n_2} \ge \varepsilon^{-2}\big\}.
\end{eqnarray*}

It is well known that the function
$$ 
  v(x;t):= \mathbb{P}\{W^x(s)\in {\cal O}, 0\le s < t\} = \mathbb{P}\{\tau_{x,{\cal O}}\ge t\}
$$
solves the initial-boundary value problem \eqref{heat}. Therefore, formula \eqref{heat1} for ${\cal O}=\Theta^1$ yields the asymptotics
\begin{equation*} 
  \mathbb{P}\{\tau_{x,\Theta^1}\ge t\} \sim \frac{4}{\pi}\,\cos(\pi x/2)  \exp(-\pi^2 t/8)
    \quad \textrm{as}\quad t\to \infty,
\end{equation*}
see \cite{Chung} or \cite[V.2, Ch. X]{Feller}. Evidently, the same result is true for the layer ${\cal O}=\Theta^1\times\bbR^{n-1}$.

In the recent paper \cite{LN} it was shown that 
$$
\mathbb{P}\{\tau_{x,\Theta^2}\ge t\} \sim {\cal A}_2 \, U^2(x)\exp\big(-\frac {\lambda_\bullet(\Theta^2)}2 t\big)    \quad \textrm{as}\quad t\to \infty.
$$
Therefore,
$$
\lim\limits_{t\to\infty} t^{-1}\log\mathbb{P}\{\tau_{x,\Theta^2}\ge t\}=-\frac {\lambda_\bullet(\Theta^2)}2>-\frac {\pi^2}8,
$$
i.e. the long stays in a two-dimensional corner or cross are more likely than
those in a strip of the same width.

Now, using the result of Theorem \ref{asymp-heat}, we can write
\begin{equation} \label{tau-n}
\mathbb{P}\{\tau_{x,\Theta^n}\ge t\} \sim {\cal A}_n \, U^n(x)\exp\big(-\frac {\lambda_\bullet(\Theta^n)}2 t\big)    \quad \textrm{as}\quad t\to \infty,
\end{equation}
therefore Theorems \ref{spectr-contin}, \ref{discrete1} and Corollary \ref{discrete2} show that
$$
\lim\limits_{t\to\infty} t^{-1}\log\mathbb{P}\{\tau_{x,\Theta^n}\ge t\}>
\lim\limits_{t\to\infty} t^{-1}\log\mathbb{P}\{\tau_{x,\Theta^{n-1}}\ge t\}.
$$

By Proposition \ref{eigenfunction}, the main term of the tail 
probability \eqref{tau-n} depends on the initial point $x$ exponentially.
So, if we start from a remote $x$, the optimal strategy to stay in $\Theta^n_1$ or in $\Theta^n_2$
for a longer time is to run towards the origin and then stay somewhere near it. This phenomenon was discovered in \cite{LN} for $n=2$.

\begin{remark}
It follows from Theorem \ref{bounded-from-zero1} and Remark \ref{bounded-from-zero2} that for any $n\ge 2$
$$
\lim\limits_{t\to\infty} t^{-1}\log\mathbb{P}\{\tau_{x,\Theta^n_1}\ge t\}\le -\frac {\pi^2}{32},
$$
while for the cross the existence of an estimate from above independent of the dimension is an open problem.
\end{remark}

\paragraph*{Acknowledgements.} We are grateful to Prof. M.A. Lifshits for valuable comments. The research of Sections \ref{contin} and \ref{discrete} was supported by the Russian Science Foundation grant 17-11-01003. The research of Sections \ref{teplo} and \ref{exit} was supported by the joint RFBR--DFG grant 20-51-12004.


\begin{thebibliography}{AFT}

\bibitem{BaMaNa} 
{\sc F.L. Bakharev, S.G. Matveenko, S.A. Nazarov}, {\em The discrete spectrum of cross-shaped waveguides}, Algebra \& Analysis, {\bf 28} (2016), N2, 58--71 (Russian); English transl.: St. Petersburg Math. J., {\bf 28} (2017), N2, 171--180.

\bibitem{BaMaNaZaa}
{\sc F.L. Bakharev, S.G. Matveenko, S.A. Nazarov}, {\em Examples of Plentiful Discrete Spectra in Infinite Spatial Cruciform Quantum Waveguides}, Z. Anal. Anwend., {\bf 36} (2017), N3, 329--341.

\bibitem{BaDebS}
{\sc R. Ba\~nuelos, R.D. DeBlassie, R.G. Smits}, {\em The first exit time of planar Brownian motion from the interior of a parabola}, Ann. Probab., {\bf 29} (2001), 882--901.

\bibitem{BaS}
{\sc R. Ba\~nuelos, R.G. Smits}, {\em Brownian motion in cones}, Probab. Theory Related Fields,
{\bf 108} (1997), 299--319.

\bibitem{BS}
{\sc M.S. Birman, M.Z. Solomjak}, {\em Spectral theory of self-adjoint operators in Hilbert space}, 2nd ed., revised and extended. Lan', St. Petersburg, 2010 [in Russian]; English transl. of the 1st ed.: Mathematics and 
Its Applic. Soviet Series. {\bf 5}, Kluwer, Dordrecht etc. 1987.

\bibitem{CaExKr2004}
{\sc G. Carron, P. Exner, D. Krej\v{c}i\v{r}\'{\i}k}, {\em Topologically nontrivial quantum layers}, J. Math. Phys., {\bf 45} (2004), N2, 774--784.

\bibitem{Chung}
{\sc K.L. Chung}, {\em On the maximal partial sums of independent random variables}, Trans. Amer. Math. Soc., {\bf 64} (1948), 205--233.

\bibitem{DaLaOu}
{\sc M. Dauge, Y. Lafranche, T. Ourmi\'eres-Bonafos}, {\em Dirichlet Spectrum of the Fichera Layer}, Integr. Equ. Oper. Theory, {\bf 90} (2018), art. no. 60, 1--33. https://doi.org/10.1007/s00020-018-2486-y

\bibitem{DaOuRa2015} 
{\sc M. Dauge, T. Ourmi\'eres-Bonafos, N. Raymond}, {\em Spectral asymptotics of the Dirichlet Laplacian in a conical layer}, Communications on Pure and Applied Analysis, {\bf 14} (2015), N3, 1239--1258.

\bibitem{DaRa2012}
{\sc M. Dauge, N. Raymond}, {\em Plane waveguides with corners the small angle limit}, J. Math. Phys., {\bf 53} (2012), N12, art. no. 123529.

\bibitem{DuEx1995}
{\sc P. Duclos, P. Exner}, {\em Curvature-induced bound states in quantum waveguides in two and three dimensions}, Rev. Math. Phys., (1995), N7, 73--102.

\bibitem{DuExKr2001}
{\sc P. Duclos, P. Exner, D. Krej\v{c}i\v{r}\'{\i}k}, {\em Bound states in curved quantum layers}, Commun. Math. Phys., {\bf 223} (2001), N1, 13--28.

\bibitem{EK}
{\sc P. Exner, H. Kova\v{r}\'{\i}k}, {\em Quantum Waveguides}, Springer, Heidelberg etc. 2015.

\bibitem{ExLo2020}
{\sc P.Exner, V. Lotoreichik}, {\em Spectral Asymptotics of the Dirichlet Laplacian on a Generalized Parabolic Layer}, Integral Equations and Operator Theory, {\bf 92} (2020), N2, art. no. 15.

\bibitem{ExTa2010} 
{\sc P. Exner, M. Tater}, {\em Spectrum of Dirichlet Laplacian in a conical layer}, Journal of Physics A: Mathematical and Theoretical, {\bf 43} (2010), N47, art. no. 474023.

\bibitem{Feller}
{\sc W. Feller}, {\em An Introduction to Probability Theory and its Applications}. Wiley, 1966.

\bibitem{Gl}
{\sc I.M. Glazman}, {\em Direct methods of qualitative spectral analysis of singular differential operators}, Fizmatlit, Moscow, 1963 [in Russian]; English transl.: Israel Program for Scient. Transl., Jerusalem, 1965.

\bibitem{GoJa1992} {\sc J. Goldstone, R.L. Jaffe}, {\em Bound states in twisting tubes},
Physical Review B, {\bf 45} (1992), N24, 14100--14107.

\bibitem{ESS}
{\sc P. Exner,  P. \v{S}eba, P. \v{S}t'ov\'{\i}\v{c}ek}, {\em On existence of a bound state in an L-shaped waveguide}. Czech. J. Phys. B, {\bf 39} (1989), 1181--1191.

\bibitem{He}
{\sc J. Hersch},  {\it Erweiterte Symmetrieeigenschaften von L\"{o}sungen gewisser linearer Rand-
und Eigenwertprobleme}. J. Reine Angew. Math., {\bf 218} (1965), 143--158.

\bibitem{K} 
{\sc I.V. Kamotski\u{\i}}, {\em Surface wave running along the edge of an elastic wedge}, Algebra \& Analysis, {\bf 20} (2008), N1, 86--92 [in Russian]; English transl.: St. Petersburg Math. J., {\bf 20} (2009), N1, 59--63.

\bibitem{WLi}
{\sc W.V. Li}, {\em The first exit time of Brownian motion from unbounded domain},
Ann. Probab., {\bf 31} (2003), N2, 1078--1096.

\bibitem{LN} 
{\sc M.A. Lifshits, A.I. Nazarov}, {\em On Brownian exit times from perturbed multi-strips}, Stat. \& Prob. Letters, {\bf 147} (2019), 1--5.

\bibitem{LifZS}
{\sc M. Lifshits, Z. Shi}, {\em The first exit time of Brownian motion from a parabolic domain}, Bernoulli, {\bf 8} (2002), N6, 745--765.

\bibitem{LiLu2007} {\sc C. Lin, Z. Lu}, {\em Existence of bound states for layers built over hypersurfaces in $R^{n + 1}$}, J. Funct. Analysis, {\bf 244} (2007), N1, 1--25.

\bibitem{LuRo2012} {\sc Z. Lu, J. Rowlett}, {\em On the discrete spectrum of quantum layers}, J. Math. Phys., {\bf 53} (2012), N7, art. no. 073519.

\bibitem{NazSA}
{\sc S.A. Nazarov}, {\em Discrete spectrum of cranked, branchy, and periodic waveguides}, Algebra \& Analysis, {\bf 23} (2011), N2, 206--247 (Russian); English transl.: St. Petersburg Math. J., {\bf 23} (2012), N2, 351--379.

\bibitem{NaRuUu2013} 
{\sc S.A. Nazarov, K. Ruotsalainen, P. Uusitalo}, {\em The $Y$-junction of quantum waveguides}, ZAMM, {\bf 94} (2014), N6, 477--486.

\bibitem{NaSh2014} {\sc S.A. Nazarov, A.V. Shanin}, {\em Trapped modes in angular joints of 2D waveguides}, Applicable Analysis,  {\bf 93} (2014), N3, 572--582.
	
\bibitem{OuPa2018} {\sc T. Ourmi\'eres-Bonafos, K. Pankrashkin}, {\em Discrete spectrum of interactions concentrated near conical surfaces},  Applicable Analysis, {\bf 97} (2018), N9, 1628--1649.

\bibitem{Pa2017}
{\sc K. Pankrashkin}, {\em Eigenvalue inequalities and absence of threshold resonances for waveguide junctions}, J. Math. Analysis Appl., {\bf 449}, 2017, N1, 907--925.

\bibitem{PT}
{\sc Y. Pinchover, K. Tintarev}, {\em Existence of minimizers for Schr\"odinger operators under domain perturbations with applications to Hardy inequality}, Indiana Univ. Math. J. {\bf 54} (2005), 1061--1074.

\bibitem{SRW}
{\sc R.L. Schult,  D.G. Ravenhall, H.W. Wyld}, {\em Quantum bound states in a classically unbounded system of crossed wires}, Physical Review B, {\bf 39} (1989), 5476--5479.

\bibitem{Shn}
{\sc \`{E}.\`{E}. \v{S}nol'}, {\em On the behavior of the eigenfunctions of Schr\"{o}dinger's equation}, Mat. Sb. (N.S.) {\bf 42 (84)} (1957), N3, 273--286 [in Russian].


\end{thebibliography}
\end{document}